\documentclass[10pt]{article}%
\usepackage{amsmath, amsthm}
\usepackage{graphicx}
\usepackage{amsfonts}
\usepackage{amssymb}
\usepackage{times}
\usepackage{amsmath}%
\usepackage{subcaption}

\providecommand{\U}[1]{\protect\rule{.1in}{.1in}}
\providecommand{\U}[1]{\protect\rule{.1in}{.1in}}
\newtheorem{theorem}{Theorem}

\begin{document}

\begin{center}
\vspace{0.4in}

{\LARGE {\textbf{Generalized Attracting Horseshoes and Chaotic Strange
Attractors}} }

\vspace{0.2in}

\textsf{Yogesh Joshi}

\textsf{Department of Mathematics and Computer Science}

\textsf{Kingsborough Community College}

\textsf{Brooklyn, NY 11235-2398}

\textsf{yogesh.joshi@kbcc.cuny.edu}

$\ast\;\ast\;\ast$

\textsf{Denis Blackmore}

\textsf{Department of Mathematical Sciences and}

\textsf{Center for Applied Mathematics and Statistics}

\textsf{New Jersey Institute of Technology}

\textsf{Newark, NJ 07102-1982}

\textsf{deblac@m.njit.edu}

$\ast\;\ast\;\ast$

\textsf{Aminur Rahman}

\textsf{Department of Applied Mathematics}

\textsf{University of Washington}

\textsf{arahman2@uw.edu}
\end{center}

\vspace{0.2in}

\noindent\textbf{ABSTRACT:} A generalized attracting horseshoe is introduced
as a new paradigm for describing chaotic strange attractors (of arbitrary
finite rank) for smooth and piecewise smooth maps $f:Q\rightarrow Q$, where
$Q$ is a homeomorph of the unit interval in $\mathbb{R}^{m}$ for any integer
$m\geq2$. The main theorems for generalized attracting horseshoes are shown to
apply to H\'{e}non and Lozi maps, thereby leading to rather simple new chaotic
strange attractor existence proofs that apply to a range of parameter values
that includes those of earlier proofs. In particular, it is shown that the
H\'{e}non map has a chaotic strange attractor for the popular parameter values
$a=1.4$, $b=0.3$, which apparently resolves an open problem.

\bigskip

\noindent\textbf{Keywords:} Generalized attracting horseshoes, Chaotic strange
attractors, Birkhoff--Moser--Smale theory, H\'{e}non attractor, Lozi attractor

\medskip

\noindent\textbf{AMS Subject Classification 2010:} 37D45; 37E99; 92D25; 92D40

\section{Introduction}

Ever since Lorenz \cite{Lor} reported on what were then surprising results of
the numerical simulations of his simplified, mildly nonlinear, 3-dimensional
ODE model of atmospheric flows, chaotic strange attractors have been the
subject of intense investigation by theoretical and applied dynamical systems
researchers. Later 3-dimensional ODE models such as R\"{o}ssler's \cite{Ros}
analytic system and Chua's \cite{CKM} piecewise linear system further
intensified the interest in these intriguing and important types of dynamical phenomena.

Over fifty years of dedicated mathematical, scientific, engineering and
economic research using analytical, computational and experimental methods has
firmly established that the identification and characterization of chaotic
strange attractors is important for both theory and applications (see
\cite{Arr,BRTUZ,CA,Cvit1,GH,HM,HKLN,JB,KH,HLor,Mil1,Mil2,OS,Rob,Ru,Shub,Smale,Wig}%
). However, this tends to be very difficult to achieve rigorously, especially
for continuous dynamical systems, as evidenced by the fact that many
properties of the Lorenz, R\"{o}ssler and Chua attractors that seem clear from
very precise and extensive simulations, have yet to be proven.

As discrete dynamical systems are generally easier to analyze than their
continuous (ODE or PDE) counterparts, their chaotic strange attractors should
be considerably easier to describe rigorously, which likely was part of the
motivation for the development of an approximate Poincar\'{e} map for the
Lorenz system by H\'{e}non \cite{Hen1,Hen2}. The H\'{e}non map is a
(quadratic) polynomial diffeomorphism of the plane with iterates that converge
to what appears to be a boomerang-shaped attractor, as shown in Fig.1, having
the look of a fractal set with a Hausdorff dimension slightly greater than one
(see \cite{Falc}), and it actually rather closely resembles the simulated
Poincar\'{e} map iterates of the Lorenz system for the right choice of a transversal.

\begin{figure}[ptb]
\centering
\includegraphics[width=0.5\textwidth]{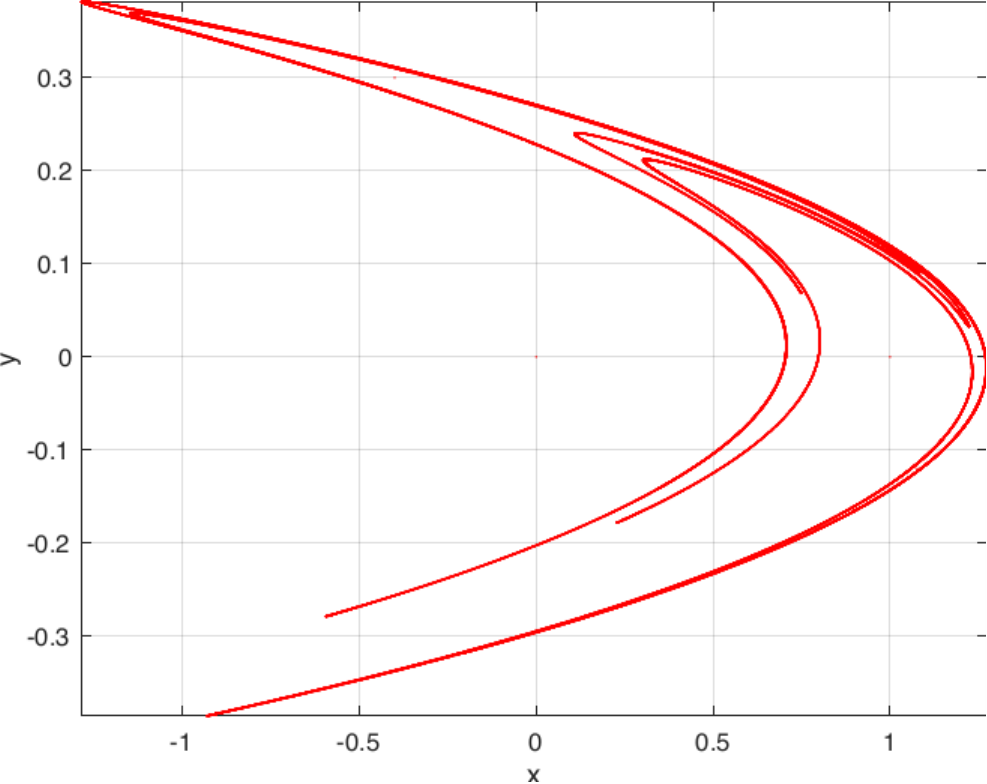}\caption{The H\'{e}non
attractor for $H(x,y)=(1-ax^{2}+y,bx)$ for $a=1.4$, $b=0.3$}%
\label{fig:Hen}%
\end{figure}

Even though the H\'{e}non map dynamics is considerably more amenable to
analysis than the Lorenz equations, proving the existence of a chaotic strange
attractor still turned out to be formidable. Possibly motivated by an interest
in finding an even simpler approximation than H\'{e}non's, Lozi
\cite{Loz1,Loz2} devised his almost everywhere analytic piecewise linear map
that appeared to have a chaotic strange attractor, illustrated in Fig.2,
resembling a piecewise linear analog of the one generated by $H$. But a
rigorous verification of the existence of the putative chaotic strange
attractor for the Lozi map proved to be quite difficult. Finally, in a
pioneering investigation involving a rather lengthy proof, Misiurewicz
\cite{Mis} proved that the Lozi map has a chaotic strange attractor for a
range of the nonnegative parameters $a$ and $b$, which to our knowledge is the
first such rigorous verification for any dynamical system. The H\'{e}non
system proved to be a much tougher nut to crack, and it was not until eleven
years later that Benedicks \& Carleson \cite{BC} proved in an almost herculean
effort involving long, detailed and subtle analysis, that the Henon map has a
chaotic strange attractor for parameters near $a=2$ and $b=0$, respectively.

The fascinating record of dedicated research offers compelling evidence of the
difficulty in proving the existence of strange attractors, even for relatively
simple nonlinear maps. This also includes attractors that simply that display
unusually high orders of complexity such as in \cite{BKNS,Zas} and many of the
books cited in our reference list, and this includes strange (fractal)
attractors that are not chaotic such as in \cite{GOPY}. At any rate, the
almost overwhelming number and diversity of chaotic strange attractors points
to a real need to develop a theory or theories that subsume significant
subclasses of these elusive and consequential dynamical entities; and in this
progress is being made. In recent years, the basic ideas behind the proofs of
the landmark Lozi and H\'{e}non map results in strange attractor theory have
been extended and generalized in terms of a theory of \emph{rank one maps} in
an extraordinary series of papers by Wang \& Young \cite{WY1,WY2,WY3}, the
content of which gives striking confirmation of the exceptional complexity
underlying characterizations of strange attractors for broad classes of
discrete dynamical systems. However, the foundational results of rank one
theory are generally hard to prove, and they tend to be rather difficult to
apply, as for example in Ott \& Stenlund \cite{OS}, which is closely related
to results of Zaslavsky \cite{Zas} and Wang \& Young \cite{WY2}. In light of
this rather daunting rigorous chaotic strange attractor landscape, it is clear
that there is a need for simpler theories and methods for reasonably ample
classes of discrete dynamical systems of significant theoretical and applied
interest. Our hope is that the work in this paper is a useful step in that direction.

\begin{figure}[ptb]
\centering
\includegraphics[width=0.5\textwidth]{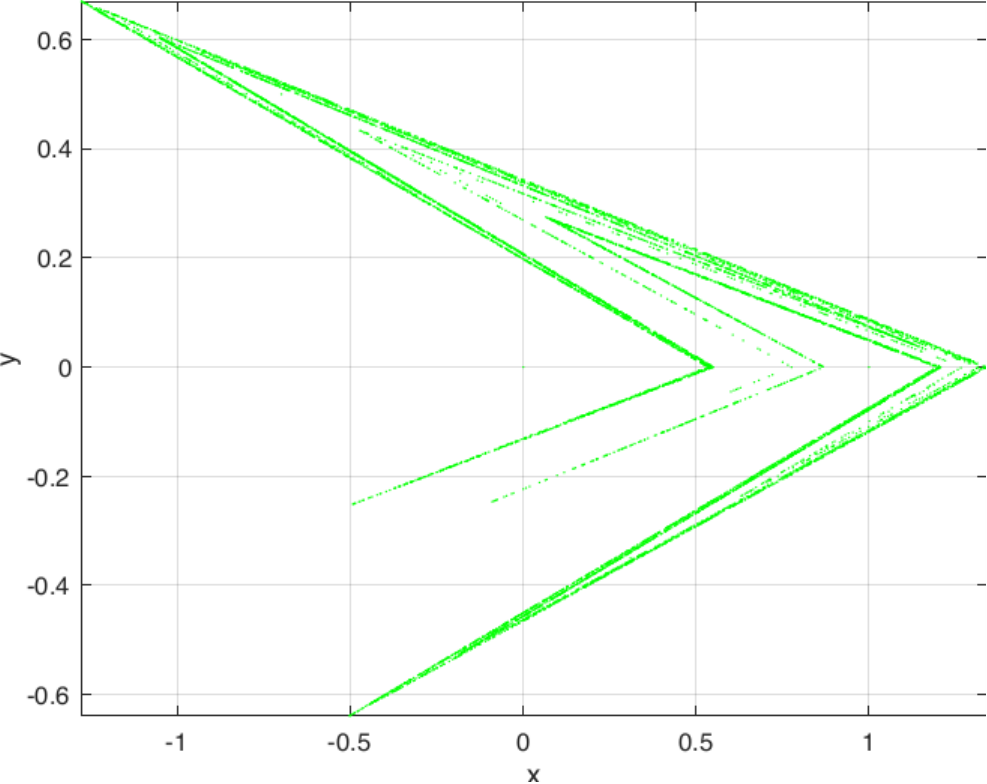}\caption{The Lozi attractor for
$L(x,y)=(1-a\left\vert x\right\vert +y,bx)$ for $a=1.7$, $b=0.5$}%
\label{fig:Loz}%
\end{figure}

The organization of the remainder of this paper is as follows. In Section 2 we
describe some notation and definitions to be used throughout the exposition.
This might seem unnecessary to experts, but since there are a number of
competing definitions that are widely accepted, it is prudent to be very
specific in certain cases. Then, in Section 3, we describe our generalized
attracting horseshoe paradigm and prove the main theorem about
its\emph{\ strange chaotic attractor}. In addition, we show how related maps
are subsumed by the paradigm, which leads to analogous chaotic strange
attractor theorems. Moreover, we extend these results to higher dimensions,
thereby obtaining rank $k$ attractors for any integer $k\geq2$. After this, in
Section 4, we apply the results in Section 3 to the H\'{e}non and Lozi maps to
obtain surprisingly short and efficient chaotic strange attractor existence
theorems for each of these discrete dynamical systems. This includes analyses
of natural extensions of the H\'{e}non and Lozi maps to $\mathbb{R}^{3}$,
focusing on their rank-2 attractors. The exposition concludes in Section 5
with a summary of our results and their impacts, as well as a brief
description envisaged related future work.

\section{Notation and Definitions}

We shall be concerned here primarily with discrete (semi) dynamical systems
generated by the nonnegative iterates of continuous maps
\begin{equation}
f:U\rightarrow\mathbb{R}^{m}, \label{eq1}%
\end{equation}
where $m\geq2$, $U$ is a connected open subset of $\mathbb{R}^{m}$, $f$ is
$C^{1}$ except possibly on finitely many submanifolds of positive codimension
having a union that does not contain any of the fixed points of the map. Our
focus shall be on planar maps ($m=2$), but we are going to consider higher
dimensions in the sequel. More specifically, we are going to concentrate on
maps of the form (\ref{eq1}) having the additional property that the is a
homeomorph of the unit disk $B_{1}(0):=\{\boldsymbol{x}\in$ $\mathbb{R}%
^{m}:\left\Vert \boldsymbol{x}\right\Vert \leq1\}$ contained in $U$, which we
denote as $Q$, such that
\begin{equation}
f\left(  Q\right)  \subset Q. \label{eq2}%
\end{equation}
Employing the usual abuse of notation, we identify the restriction of $f$ to
$Q$ with $f$ itself, so that our primary concern is with the dynamics of the
maps
\begin{equation}
f:Q\rightarrow Q \label{eq3}%
\end{equation}
subject to the above assumptions. We denote this set of maps as $\mathfrak{F}%
^{1}(Q)$ and remark that we have included the possibility of maps that may not
be differentiable on sets of (Lebesgue) measure zero because we are going to
analyze Lozi attractors.

Our aim is to identify and characterize attractors for maps of the form
(\ref{eq3}); in particular attractors that are fractal sets on which the
restricted dynamics exhibit chaotic orbits. For our definition of chaos, we
take the description ascribed to Devaney, which requires sensitive dependence
on initial conditions, density of the set of periodic points and topological
transitivity, keeping in mind that Banks \emph{et al.} \cite{BBCDS} proved
that sensitive dependence is implied by periodic density and transitivity. For
more standard definitions, we refer the reader to
\cite{Arr,Dev,GH,HKLN,KH,Rob,Shub,Wig}.

We are now ready to give a precise definition of a \emph{chaotic strange
attractor} (\emph{CSA}) - our principal object of interest.

\bigskip

\noindent\textbf{Definition} Let $f\in$\noindent$\mathfrak{F}^{1}(Q)$ and
$\mathfrak{A}\subset Q$. Then $\mathfrak{A}$ is a chaotic strange attractor
for $(3)$ if it satisfies the following properties:

\begin{itemize}
\item[.]

\begin{itemize}
\item[(CSA1)] it is a compact, connected,$f$-invariant subset of $Q.$

\item[(CSA2)] $\mathfrak{A}$ is an attractor in the sense that there is an
open set $V$ of $\mathbb{R}^{m}$ such that $A\subset V\cap Q$ and $d\left(
f^{n}(\boldsymbol{x}),\mathfrak{A}\right)  \rightarrow0$ as $n\rightarrow
\infty$ for every $x\in V\cap Q$.

\item[(CSA3)] it is the minimal set satisfying CAS1 and CAS2.
\end{itemize}
\end{itemize}

\section{Attracting Horseshoes and Their Generalizations}

Attracting horseshoes (AH) were introduced in Joshi \& Blackmore \cite{JB} as
a CSA model that can be extended to any finite rank. The CSA for the AH can be
readily shown to be given by
\[
\mathfrak{A}:=\overline{W^{u}(p)}.
\]
As one can plainly see in Fig.3, these AHs are basically the standard Smale
horseshoes described in such treatments as \cite{Arr,DN,GH,KH,Shub,Smale,Wig},
which were also employed by Easton \cite{East} in his work on trellises.
Attracting horseshoes have precisely three fixed points comprising two saddles
points $p$ and $q$ and one sink $r$, such that $f(A)\subset B$, with $B$
contained in the basin of attraction of the sink. It should be noted that
since $f(Q)$ is contained in the interior of $Q$, which is diffeomorphic to a
disk in the plane, it is a simple matter to extend $f$ to a diffeomorphism
$F:\mathbb{R}^{2}\rightarrow\mathbb{R}^{2}$ having $\overline{W^{u}(p)}$ as a
global attractor. The main novelty in \cite{JB} was the focus on attractors,
and especially multihorseshoe chaotic strange attractors produced by AHs for
iterated maps, which are apt to display extraordinary complexity. Inasmuch as
AHs have three fixed points - two saddles and one sink - these models are not
suitable for the analysis of maps such as those of H\'{e}non and Lozi, which
have only two fixed points, both of which are saddle points. It was precisely
this observation that led to our development of the \emph{generalized
attracting horseshoe} (\emph{GAH}).

\begin{figure}[ptb]
\centering
\includegraphics[width=0.6\textwidth]{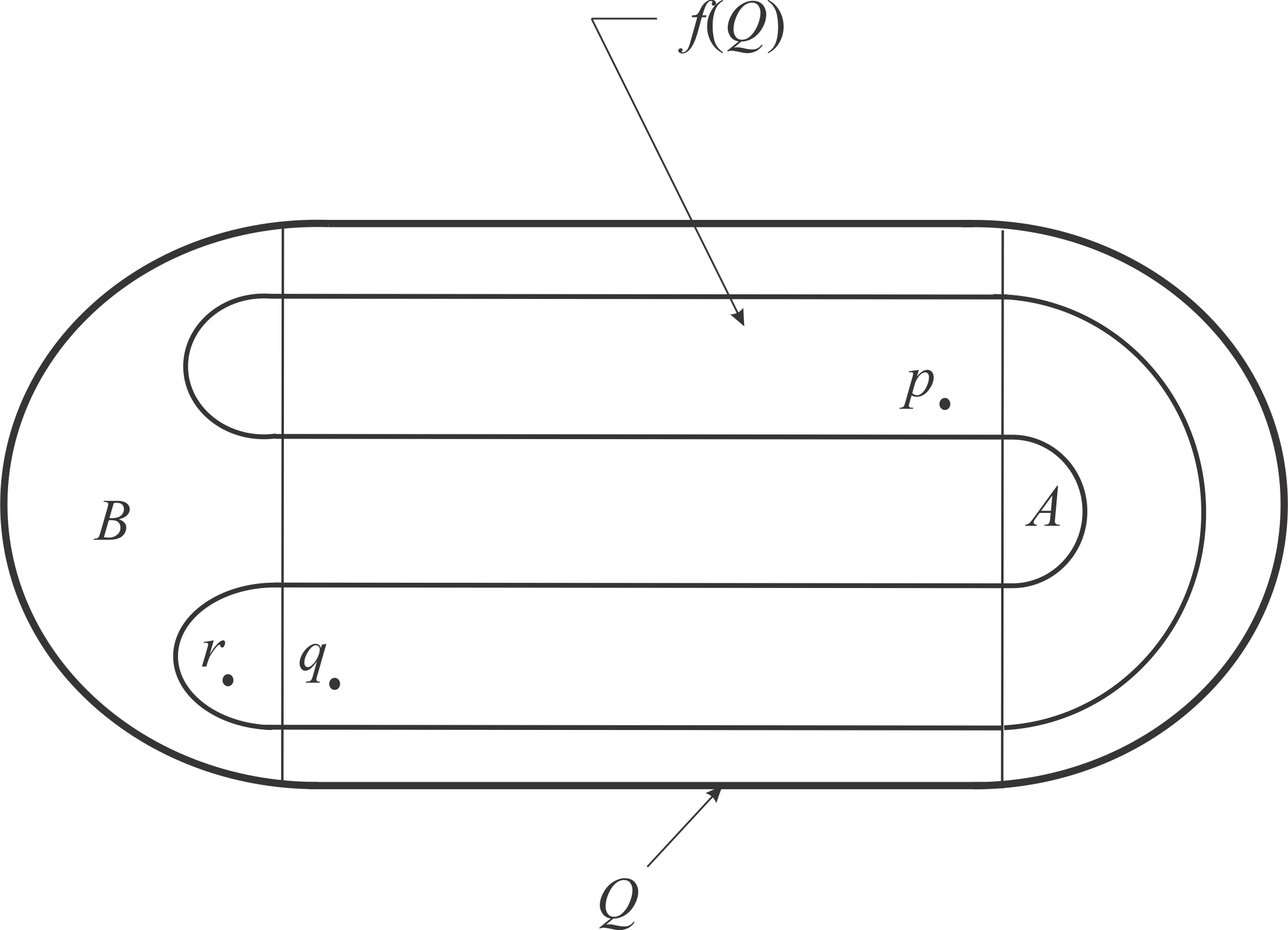}\caption{AH model}%
\label{fig:AH}%
\end{figure}

\subsection{The generalized attracting horseshoe (GAH)}

The GAH is a modification of the AH that can be represented as a geometric
paradigm with either just one or two fixed points, both of which are saddles.
As a result, we shall show that it includes both the H\'{e}non and Lozi maps,
which leads to simple proofs of the existence of CSAs for these maps that are
essentially simple applications of the main theorems that we shall attend to
in this subsection.

Figure 4 shows a rendering of a $C^{1}$ GAH with two saddle points, which can
be constructed as follows: The rectangle is first contracted vertically by a
factor $0<\lambda_{v}<1/2$, then expanded horizontally by a factor
$1<\lambda_{h}<2$ and then folded back into the usual horseshoe shape in such
a manner that the total height and width of the horseshoe do not exceed the
height and width, respectively of the rectangle $Q$. Then the horseshoe is
translated horizontally so that it is completely contained in $Q$. Obviously,
the map $f$ defined by this construction is a member of $\mathfrak{F}^{1}(Q)$
defined in Section 2. Clearly, there are also many other ways to obtain this
geometrical configuration. For example, the map $f$ as described above is
orientation-preserving, and an preserving variant can be obtained by composing
it with a reflection in the horizontal axis of symmetry of the rectangle, or
by composing it with a reflection in the vertical axis of symmetry followed by
a composition with a half-turn. Another construction method is to use the
standard Smale horseshoe that starts with a rectangle, followed a horizontal
composition with just the right scale factor or factors to move the image of
$Q$ into $Q$, while preserving the expansion and contraction of the horseshoe
along its length and width, respectively. Note that the subrectangle $S$ with
its left vertical edge through $p$, which contains the arch of the horseshoe
and the keystone region $K$, is to play a key role in our main theorems, which follow.

\begin{theorem}
\label{thm1} Let $f:Q\rightarrow Q$ be the member of $\mathfrak{F}^{1}(Q)$
representing the GAH paradigm with a horseshoe-like image as shown in Fig. 4.
Then if $f$ satisfies the additional property

\medskip

\noindent$(T)$ f maps the keystone region $K$ $($containing a portion of the
arch of the horseshoe$)$ to the left of the fixed point $p$, i.e. $f(K)\subset
Q\smallsetminus S=\{(x,y)\in Q:x<x(p)\}$,
\[
\mathfrak{A}:=\overline{W^{u}(p)}%
\]
is a CSA.
\end{theorem}

\begin{proof}It follows from the construction that $\overline{W^{u}(p)}$ is a
compact attractor, so it remains to prove that it is strange (fractal) and has chaotic orbits.
But this is precisely where the property $(T)$ comes into play. For it guarantees that a strip
(tubular neighborhood) around the unstable manifold of $f^2$ at $p$ completely crosses a strip
around the stable manifold of $f^2$ at $p$ as shown in Fig. 5, which presents a magnified picture
of the image $f^2\left( Q\right)$ in a neighborhood of $p$.
Hence, we conclude from the well-known results of Birkhoff--Moser--Smale (\emph{cf.} \cite{Moser,Smale,Wig}),
and a slight generalization for the crossing of stable and unstable manifolds is not necessarily transverse
(see, \emph{eg.} \cite{JB}), that the attractor locally has the structure of the Cartesian product of a Cantor
middle-third set and an interval, and exhibits (symbolic) shift map chaotic dynamics.
\end{proof}

\begin{figure}[tbh]
\centering
\includegraphics[width=0.5\textwidth]{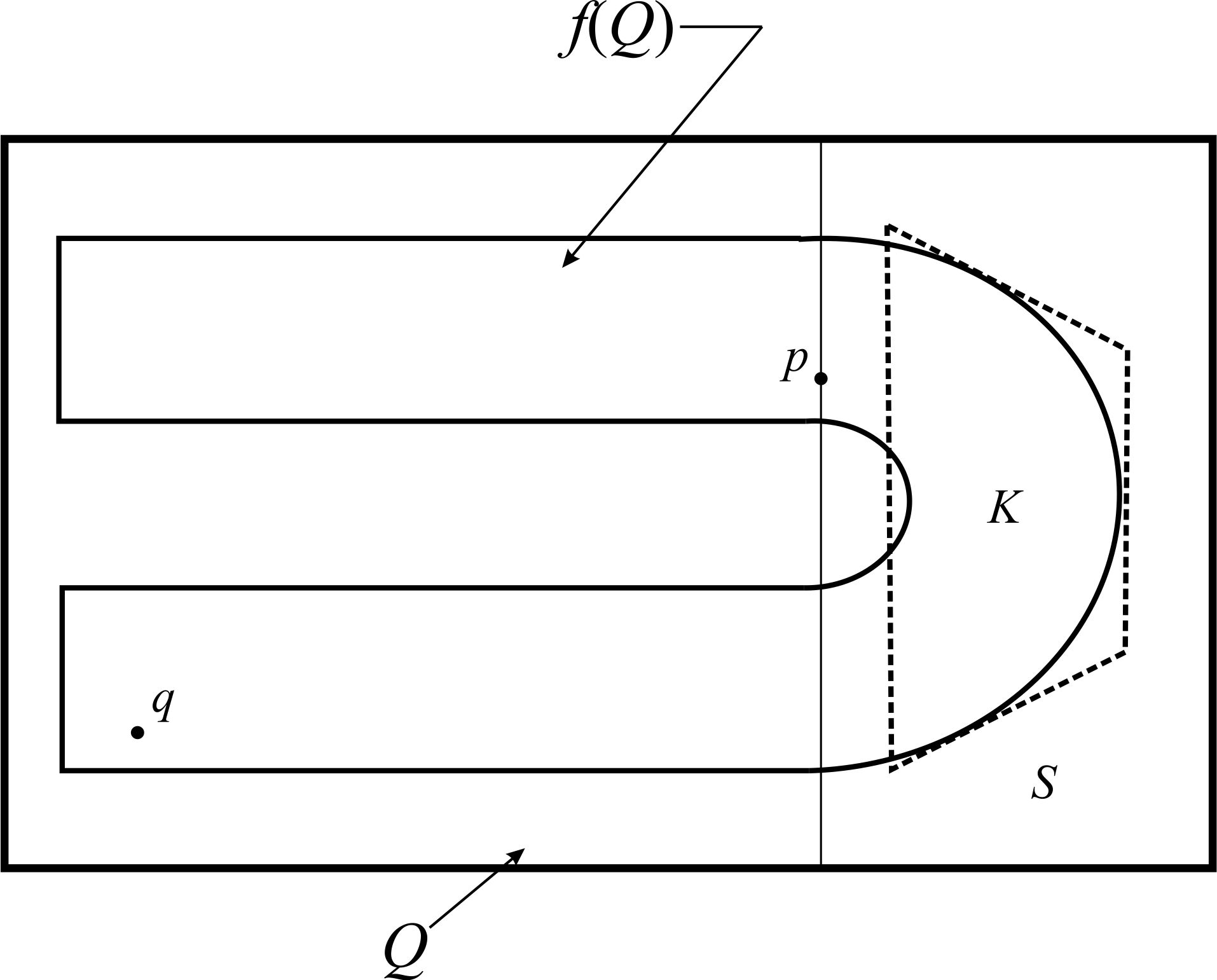}\caption{A planar GAH with two
saddle points}%
\label{fig:GAH}%
\end{figure}

Theorem \ref{thm1} can readily be adapted to cover even more general
horseshoe-like maps in $\mathfrak{F}^{1}(Q)$, with virtually the same proof.

\begin{theorem}
\label{thm2} Let $f:Q\rightarrow Q$ be any member of $\mathfrak{F}^{1}(Q)$
with a horseshoe-like image with a keystone region $K$ containing a portion of
the arch of $f(Q)$ analogous to that shown in Fig.4. Suppose that the map is
expanding by a scale factor uniformly greater than one along the length of the
horseshoe and contracting transverse to it by a scale factor uniformly less
than one-half in the complement of a subset of Q containing K. Then if $f$
satisfies the additional property

\medskip

\noindent$(\tilde{T})$ f maps K into an open subset of Q to the left of the
saddle point p,\smallskip%
\[
\mathfrak{A}:=\overline{W^{u}(p)}%
\]
\smallskip\noindent is a CSA.
\end{theorem}

\begin{proof}This follows mutatis mutandis from the proof of Theorem \ref{thm1}.
\end{proof}

\begin{figure}[hbt]
\centering
\includegraphics[width=0.5\textwidth]{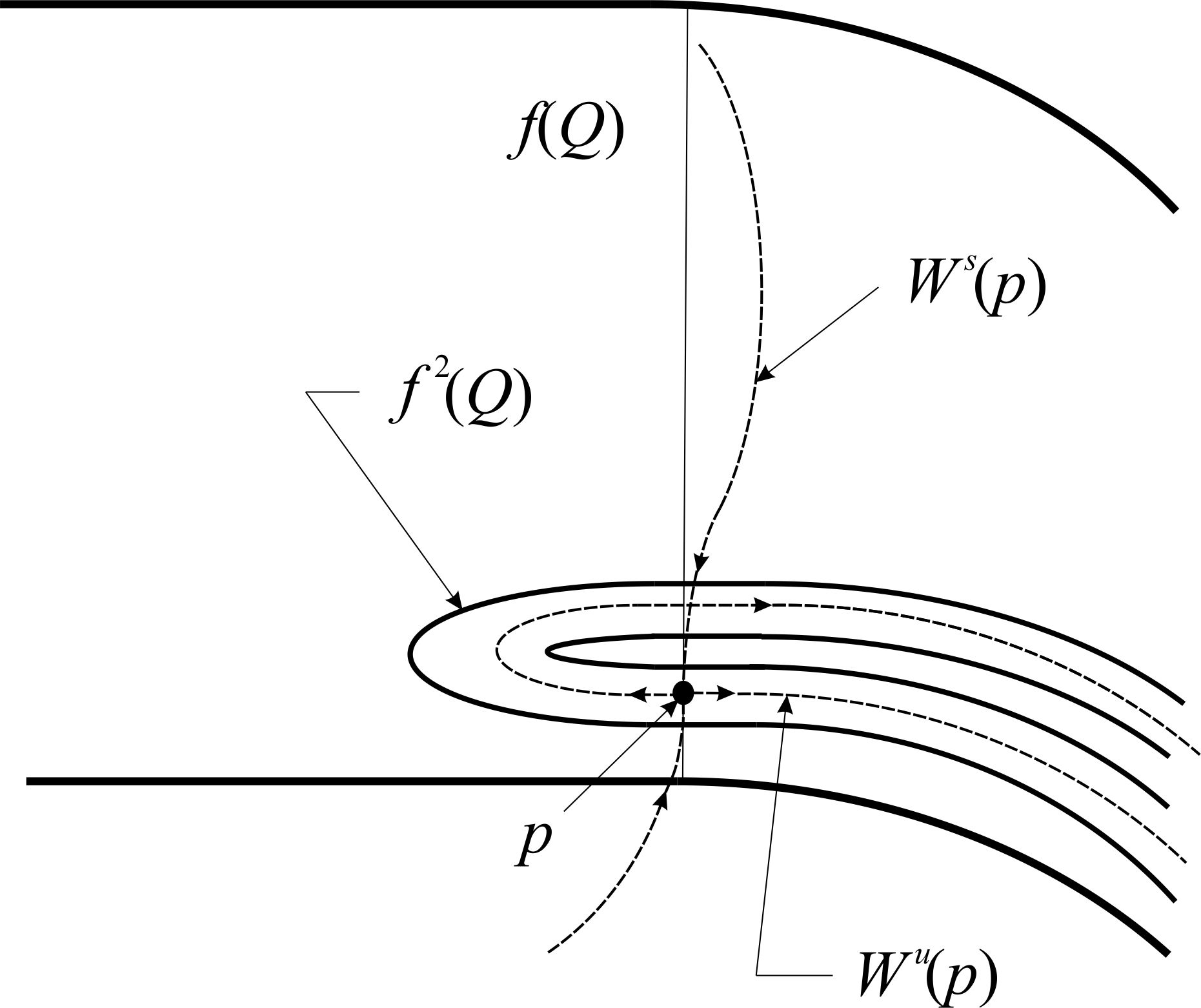}\caption{Local (transverse)
horseshoe structure of $f^{2}$ near $p$}%
\label{fig:f2p}%
\end{figure}

\subsection{Higher dimensional GAH paradigms}

The planar GAH model can be extended to any finite dimension to produce CSAs
of any rank. For example, this can be accomplished inductively by composing
the paradigm with the model map in successive coordinate planes formed by the
contracting coordinate and each new (expanding) coordinate direction that is
added. For demonstration purposes, it suffices to show how to go from a
1-dimensional unstable manifold to a 2-dimensional unstable manifold in
$\mathbb{R}^{3}$.

We may write the planar GAH in the form
\begin{equation}
f(x,y):=\left(  u(x,y),v(x,y)\right)  , \label{eq4}%
\end{equation}
which we extend to Euclidian 3-space as
\begin{equation}
f_{1}(x,y,z):=\left(  f(x,y),z\right)  =\left(  u(x,y),v(x,y),z\right)  .
\label{eq5}%
\end{equation}
Then, holding $x$ fixed and applying the planar model map in the $z$-$y$
coordinate plane corresponds to the mapping
\begin{equation}
f_{2}(x,y,z):=\left(  x,v(z,y),u(z,y)\right)  . \label{eq6}%
\end{equation}
Therefore, the desired extension is
\begin{equation}
f_{3}(x,y,z):=f_{2}\circ f_{1}(x,y,z)=f_{2}\left(  u(x,y),v(x,y),z\right)
=\left(  u(x,y),v\left(  z,v(x,y)\right)  ,u\left(  z,v(x,y)\right)  \right)
. \label{eq7}%
\end{equation}
Higher dimensional GAH paradigms and, more generally, GAH models of the type
covered in Theorem \ref{thm2} of any finite dimension can be created by
successive applications of the inductive step described above, which allows to
construct GAHs of any rank. To visualize the nature of the image of $f_{3} $,
picture a thickened plane perpendicular to the $z$-axis that is first folded
quite sharply along the $x$-axis and then folded rather more gently along the
$z$-axis.

\section{Applications to the H\'{e}non and Lozi Maps}

The existence of CSAs for the H\'{e}non and Lozi maps now turn out to be
direct corollaries of Theorem \ref{thm2}. We consider the H\'{e}non map
\begin{equation}
H(x,y):=\left(  1-ax^{2}+y,bx\right)  \label{eq8}%
\end{equation}
for a small parameter neighborhood of $(a,b)=(1.4,0.3)$ and the Lozi map
\begin{equation}
L(x,y):=\left(  1-\alpha\left\vert x\right\vert +y,\beta x\right)  \label{eq9}%
\end{equation}
in a parameter neighborhood of $(\alpha,\beta)=(1.7,0.5)$ that their
(apparent) respective CSAs at $(a,b)=(1.4,0.3)$ and $(\alpha,\beta)=(1.7,0.5)$
are illustrated in Fig. 1 and Fig. 2, respectively.

Our first result is for the H\'{e}non map (\ref{eq8}), which has a
quadrilateral trapping region $Q$, first observed in \cite{Hen2}, with
vertices at $A$: (-1.33,0.42), $B$: (1.32,0.133), $C$: (1.245,-0.14) and $D$:
(-1.06,-0.5). The map has just two fixed points, $p$ and $q$, both of which
are saddles. The fixed points are (approximately):
\[
p=(0.631354477,0.189406)\ \text{and}\ q=(-1.131354477,-0.339406343),
\]
and we note that only $p\in Q$.

\begin{theorem}
\label{thm3}The H\'{e}non map $(\ref{eq8})$ has a CSA given by
\[
\mathfrak{A}:=\overline{W^{u}(p)}%
\]
in a sufficiently small parameter neighborhood of $(a,b)=(1.4,0.3)$.
\end{theorem}

\begin{proof}
It is straightforward to prove that the trapping quadrilateral $Q$, with
vertices $A$, $B$, $C$ and $D$ given above, is such that (\ref{eq8}) is a
member of $\mathfrak{F}^{1}(Q)$ and it satisfies the hypotheses of Theorem
\ref{thm2} for a sufficiently small neighborhood of $(a,b)=(1.4,0.3)$. We
shall now verify this, leaving some of the routine details to the reader. To
begin, we show that the map satisfies property $(\tilde{T})$, which entails
proving that a properly chosen keystone set is mapped to the left of the fixed
point $p$. First we find the extremes of the intersection of $H(Q)$ with the
$x$-axis. Owing to the nature of the map, this can be done by computing the
images of the endpoints of the vertical line corresponding to $x=0$ in $Q$. It
is convenient to denote the upper and lower endpoints by $a_{+}$ and $a_{-}$,
respectively. The endpoint $a_{+}$ is the intersection point of the $y$-axis
and the edge $AB=(1-t)A+tB$ $(0\leq t\leq1)$, which we compute by solving the
following equations for $t_{+}$ and then $y:=y(a_{+})$:%
\begin{align*}
(1-t_{+})(-1.33)+t_{+}(1.32)  &  =0\Longrightarrow t_{+}=\frac{1.33}%
{2.65},\;(1-t_{+})(0.42)+t_{+}(0.133)=\\
0.42-0.287t_{+}  &  =\frac{0.73129}{2.65}=y\approx0.276\Longrightarrow
a_{+}\approx(0,0.276).
\end{align*}
Similarly, $a_{-}$ is the point of intersection of the $y$-axis and the edge
$DC=(1-t)D+tC$ $(0\leq t\leq1)$, which we determine by solving%
\begin{align*}
(1-t_{-})(-1.06)+t_{-}(1.245)  &  =0\Longrightarrow t_{-}=\frac{1.06}%
{2.305},\;(1-t_{-})(-0.5)+t_{-}(-0.14)=\\
-0.5+0.36t_{-}  &  =-\frac{0.7709}{2.305}=y(a_{-})\approx
-0.3344\Longrightarrow a_{-}\approx(0,-0.3344).
\end{align*}
Whence, we compute that $H(a_{+})\approx(1.276,0)$ and $H(a_{-})\approx
(0.6655,0)$, and more precisely that the intersection of $H(Q)$ with the
$x$-axis, comprising the `centerline' of the arch of the horseshoe, is a
closed interval contained in $(0.6655,1.276)\times\{0\}$. Consequently, we may
choose the keystone $K$ so that it is completely contained in the subset
$Q_{K}$ of $Q$ consisting of points $(x,y)$ with $x\geq x_{K}:=0.65$. It is
easy to see that from (\ref{eq8}) and the definition of the trapping
quadrilateral $Q$, that in order to show that $H(Q_{K})$ lies to the left of
the saddle point $p$, it suffices to show that he intersection point of the
vertical line $x=x_{K}$ with the edge $AB$, which we denote by $z$, is mapped
to the left of $p$ by $H$.
To compute $z$, we solve the equations%
\begin{align*}
(1-t_{z})(-1.33)+t_{z}(1.32)  &  =0.65\Longrightarrow t_{z}=\frac{1.98}%
{2.65},\;(1-t_{z})(0.42)+t_{z}(0.133)=\\
0.42-0.287t_{z}  &  =\frac{0.54474}{2.65}=y(z)\approx0.2056\Longrightarrow
z\approx(0.65,0.2056).
\end{align*}
Hence, $H(z)\approx\left(  1-1.4(0.65)^{2}+0.2056,0.3(0.65)\right)
=(0.6141,0.195)$. More precisely, the $x$-coordinate of $H(z)$ is strictly
less than $0.62<0.63<x(p)$, which verifies property $(\tilde{T}).$
Next, we shall show that $H$ is expanding along the length of the horseshoe by
focusing on the foliation of $Q$ along its width by the family of line
segments for $0\leq s\leq1$ defined as
\[
l_{s}:=\left(  -1.33+0.27s+2.65t-0.345st,0.42-0.92s-0.287t+0.647st\right)
\;(0\leq t\leq1).
\]
Note that $l_{0}$ and $l_{1}$ correspond to the edges $AB$ and $DC$,
respectively. Moreover, as the left and right endpoints of $l_{s}$ are,
respectively, $(-1.33+0.27s,0.42-0.92s)$ and $(1.32-0.075s,0.133-0.273s)$, it
follows that the length of $l_{s}$ is%
\[
\mathcal{L}\left(  l_{s}\right)  =\left(  7.104869-2.199878s+0.537634s^{2}%
\right)  ^{1/2},
\]
which is a strictly decreasing function of $s$ for $0\leq s\leq1$ such that
$2.3<\mathcal{L}\left(  l_{s}\right)  <2.7$. The images $H\left(
l_{s}\right)  $ are parabolas that foliate $H(Q)$, and we shall prove the
expansiveness along the length by showing that
\[
\mathcal{L}H\left(  l_{s}\right)  >3.2
\]
for all $s\in\lbrack0,1]$.
The Euclidean equation of each of the lines can be readily shown to have the
form%
\[
l_{s}:y=\alpha(s)+m(s)x,
\]
where the slope is
\[
m(s):=\frac{0.647s-0.287}{2.65-0.345s}%
\]
and%
\[
\alpha(s):=0.42+0.92s+m(s)(1.33-0.27s).
\]
It is useful to note that
\[
\left\vert m(s)\right\vert <0.1562
\]
and $\alpha$ is a strictly decreasing function of $s$ satisfying%
\[
0.275<\alpha(s)<1.51
\]
for all $0\leq s\leq1$. Whence, it is straightforward to show using%
\[
\xi=1-1.4x^{2}+y,\;\eta=0.3x,
\]
that the image $H\left(  l_{s}\right)  $ has, for each $s\in\lbrack0,1]$, the
following representation in the $\xi,\eta$-plane:%
\[
\xi-\left(  1+\alpha(s)+\frac{35m^{2}}{196}\right)  =-\frac{140}{9}\left(
\eta-\frac{3m}{28}\right)  ^{2},
\]
which is a parabola with axis of symmetry $\eta=3m/28$ - a horizontal line
within a distance of $0.017$ of the $\xi$-axis for all $0\leq s\leq1.$
Furthermore, since it is easy to show that
\[
H\left(  AD\right)  ,H\left(  BC\right)  \subset Q_{(-1)}:=\{(x,y)\in
Q:x<-1\},
\]
which implies that each of the above parabolic arcs have endpoints in
$Q_{(-1)}$. Moreover, all of the parabolas lie between the extremes%
\[
H\left(  l_{0}\right)  :\xi-1.278052=-\frac{140}{9}\left(  \eta
+0.011604\right)  ^{2}%
\]
and
\[
H\left(  l_{1}\right)  :\xi-0.669909018=-\frac{140}{9}\left(  \eta
-0.016733808\right)  ^{2}%
\]
so their maximal $\xi$ values are all greater than $0.6.$ Hence, the positions
of their endpoints implies that $\mathcal{L}H\left(  l_{s}\right)  >3.2$ for
all $s\in\lbrack0,1]$, which proves that the map is expanding along the horseshoe.
One nice feature of the geometric argument above is that it can also be used
to verify the necessary transverse contraction property. Indeed, it is easy to
show that $H$ is contracting by a factor of absolute value less than $0.5$
along the vertical in $\xi,\eta$-plane in the complement of the `keystone' set
described in the first part of our proof. Thus, the existence of the CSA for
the map (\ref{eq8}) with $(a,b)=(1.4,0.3)$ of the specified type is guaranteed
by Theorem \ref{thm2}. Finally, noting that all of the elements used in the
above argument, including the vertices of the trapping set, the position and
type of the fixed points and the vertices of the keystone set are all
continuous functions of the parameters in a sufficiently small neighborhood of
$(a,b)=(1.4,0.3)$, we obtain the desired result and the proof is complete.
\end{proof}

\noindent We note here that the above result appears to solve an open problem
concerning the existence of a chaotic strange attractor for the specified
parameter values. Regarding the fractal nature of the CSA, numerical
simulation methods have been used in \cite{RHO} to show that the Hausdorff
dimension of the H\'{e}non attractor for $(a,b)=(1.4,0.3)$ is approximately
equal to $1.26$.

The existence proof for the Lozi map (\ref{eq9}) also follows with ease. It
can be readily shown to have a triangular trapping region $Q$ with vertices at
$A$: (-1.28392467,0.671742549), $B$: (1.3434851,0) and $C$:
(-0.51092939,-0.641962335), and just two fixed points, both saddles; namely,%

\[
p=(0.4545454,0.2272727)\ \text{and}\ q=(-0.8333333,-0.4166666),
\]
with only $p \in Q$.

\begin{theorem}
\label{thm4}The Lozi map $(\ref{eq9})$ has a CSA given by
\[
\mathfrak{A}:=\overline{W^{u}(p)}%
\]
in a sufficiently small parameter neighborhood of $(\alpha,\beta)=(1.7,0.5)$.
\end{theorem}

\begin{proof}
It can readily be proved that (9) is for the triangle (a degenerate
quadrilateral) $Q$, with vertices $A$, $B$, and $C$ defined above, a member of
$\mathfrak{F}^{1}(Q)$ (failing to be differentiable only along the $y$-axis)
that satisfies the hypotheses of Theorem \ref{thm2} for a sufficiently small
neighborhood of $(\alpha,\beta)=(1.7,0.5)$. We show this, omitting some
routine detail,s in a manner analogous that used to prove Theorem \ref{thm3}.
We begin by showing that the map satisfies property $(\tilde{T})$, which
requires demonstrating that a properly chosen keystone set is mapped to the
left of the fixed point $p$. The first step is to compute the extremes of the
intersection of $L(Q)$ with the $x$-axis. It follows from the definition of
the map $L$ that this can be done by computing the images of the endpoints of
the vertical line corresponding to $x=0$ in $Q$. As in the proof of Theorem
\ref{thm3}, we denote the upper and lower endpoints by $a_{+}$ and $a_{-}$,
respectively. Since the endpoint $a_{+}$ is the intersection point of the
$y$-axis and the edge $AB=(1-t)A+tB$ $(0\leq t\leq1)$, we need to solve the
following equations for $t_{+}$ and then $y:=y(a_{+})$:%
\begin{align*}
(1-t_{+})(-1.28392467)+t_{+}(1.3434581)  &  =0\Longrightarrow t_{+}%
=\frac{1.28392467}{2.62740977},\;(1-t_{+})(0.671742549)+t_{+}(0)=\\
\frac{0.902405167}{2.62740977}  &  =y\approx0.3434581\Longrightarrow
a_{+}\approx(0,1.3434581).
\end{align*}
Analogously, $a_{-}$ is the point of intersection of the $y$-axis and the edge
$CB=(1-t)C+tB$ $(0\leq t\leq1)$, which can be determined by solving%
\begin{align*}
(1-t_{-})(-0.51092939)+t_{-}(1.3434851)  &  =0\Longrightarrow t_{-}%
=\frac{0.51092939}{1.85441449},\;(1-t_{-})(-0.641962335)+t_{-}(0)=\\
-\frac{0.862466831}{1.85441449}  &  =y(a_{-})\approx-0.46509\Longrightarrow
a_{-}\approx(0,-0.46509).
\end{align*}
Consequently, we compute that $L(a_{+})\approx(1.34349,0)$ and $L(a_{-}%
)\approx(0.53491,0)$, and a closer examination of the computations shows that
the intersection of $L(Q)$ with the $x$-axis, comprising the `centerline' of
the pointed arch of the horseshoe, is a closed interval contained in
$(0.53491,1.34349)\times\{0\}$. Accordingly we may choose the keystone $K$ so
that it is completely contained in the subset $Q_{K}$ of $Q$ consisting of
points $(x,y)$ with $x\geq x_{K}:=0.5$. It is readily deduced from (\ref{eq9})
and the definition of the trapping triangle $Q$, that in order to show that
$L(Q_{K})$ lies to the left of the saddle point $p$, it suffices to show that
he intersection point of the vertical line $x=x_{K}$ with the edge $AB$, which
we denote by $z$, is mapped to the left of $p$ by $L$.
The value of $z$ can be obtained by solving the following equations:%
\begin{align*}
(1-t_{z})(-1.28392467)+t_{z}(1.3434851)  &  =0.5\Longrightarrow t_{z}%
=\frac{1.78392467}{2.62740977},\;(1-t_{z})(0.671742549)+t_{z}(0)=\\
\frac{0.566604831}{2.62740977}  &  =y(z)\approx0.2157\Longrightarrow
z\approx(0.5,0.2157).
\end{align*}
Hence, $L(z)\approx\left(  1-1.7\left\vert 0.5\right\vert
+0.2157,0.5(0.5)\right)  =(0.3657,0.195)$. And a more careful examination of
the calculations shows that the $x$-coordinate of $L(z)$ is strictly less than
$0.37<0.45<x(p)$, which proves that property $(\tilde{T})$ is satisfied$.$
We shall next verify that $L$ is expanding along the length of the horseshoe
and that it is contracting in a transverse direction. Rather than taking a
geometric approach analogous to that used in the proof of Theorem \ref{thm3},
we shall use a method that takes advantage of the fact that the derivative of
$L$, where it exists, is piecewise constant. Note that, in terms of the
standard matrix representation, we have
\[
L^{\prime}(x,y)=\left\{
\begin{array}
[c]{cc}%
\mathcal{D}_{-}, & x<0\\
\mathcal{D}_{+}, & x>0
\end{array}
\right.  ,
\]
where
\[
\mathfrak{D}_{-}:=\left(
\begin{array}
[c]{cc}%
1.7 & 1\\
0.5 & 0
\end{array}
\right)  \text{ and }\mathfrak{D}_{+}:=\left(
\begin{array}
[c]{cc}%
-1.7 & 1\\
0.5 & 0
\end{array}
\right)  .
\]
Consequently, the scale factor of the map $L$ along any direction specified by
the unit vector $\boldsymbol{\hat{u}}(\theta):=(\cos\theta,\sin\theta)$ is
given either by
\[
\sigma_{-}(\theta):=\left\vert \left\langle \mathfrak{D}_{-}\boldsymbol{\hat
{u}}(\theta),\boldsymbol{\hat{u}}(\theta)\right\rangle \right\vert
=1.7\cos^{2}\theta+0.75\sin2\theta\text{ or }\sigma_{-}(\theta):=\left\vert
\left\langle \mathfrak{D}_{-}\boldsymbol{\hat{u}}(\theta),\boldsymbol{\hat{u}%
}(\theta)\right\rangle \right\vert =-1.7\cos^{2}\theta+0.75\sin2\theta.
\]
To verify the expanding nature of the map along the horseshoe, we consider the
(singular) foliation of $Q$ by the rays emanating for $B$ that are contained
in the triangle between the two extremes, $AB$ and $CB$, with the slopes
linearly increasing from the upper to the lower edge. The corresponding range
of angles (in degrees) is from
\[
\tan\theta_{AB}=\frac{\Delta y}{\Delta x}=\frac{-0.671742549}{2.62740977}%
\approx-0.255667219\Longrightarrow\theta_{AB}\approx-14.34144088^{o}%
\]
to%
\[
\tan\theta_{CB}=\frac{\Delta y}{\Delta x}=\frac{0.641962335}{1.85441449}%
\approx0.346180607\Longrightarrow\theta_{CB}\approx19.09486124^{o}.
\]
Whence, we find for the interval $\left[  -14.34144088^{o}
,19.09486124^{o}\right] $ that
\[
\min\sigma_{-}(\theta)\geq1.231\text{ and }\min\sigma_{-}(\theta)\geq1.054.
\]
As both of these minima exceed one, and we can ignore the behavior along $x=0$
owing to the fact that it has Lebesgue measure zero, the map is expanding
along the horseshoe.
It remains to prove that the map is contracting transverse to the horseshoe.
The vertical direction should be avoided because $L$ maps small vertical line
segments to the left and right of the $y$-axis into horizontal line segments.
However, any near vertical (singular) foliation can be used. In this regards,
a simple calculation along the lines of those above shows that
\[
\max\sigma_{-}(\theta),\max\sigma_{-}(\theta)\leq1/2
\]
whenever $75^{o}\leq\theta\leq105^{0}$, which verifies the transverse
contraction of the map $L.$
Thus, it follows from Theorem \ref{thm2} that the theorem is proved for
$(\alpha,\beta)=(1.7,0.5)$. Therefore, owing to the local continuous dependence
on of all the key features used in the proof such as the position and type
of the fixed point $p$ and the vertices of the trapping triangle $Q$, the
desired result follows, which completes the proof.
\end{proof}

The above theorems provide a great deal of information about the long-term
dynamics of the H\'{e}non and Lozi maps, which when combined with
Cvitanovi\'{c}'s pruning techniques and kneading theory such as in
\cite{Cvit1,Cvit2,Ish} ought to reveal a great deal about the nature of the
iterates of the mappings. In this regard, see also \cite{DN}. We note that it
appears that the above theorems for the H\'{e}non and Lozi maps can also be
proved using rank one theory, possibly with minor modifications. However, the
rather obvious extensions to higher dimensional maps of these results, which
can be proved using an analogous combination of geometric and analytic tools,
\ are not directly amenable to rank one theory verification.

\subsection{Higher dimensional H\'{e}non and Lozi maps}

Three-dimensional extensions of the H\'{e}non and Lozi maps can be easily
obtained by the inductive process developed in Section 3. In particular, the
extensions to $\mathbb{R}^{3}$ are%
\begin{equation}
H_{3}(x,y,z):=\left(  1-ax^{2}+y,\bar{b}z,1-az^{2}+bx\right)  , \label{eq10}%
\end{equation}
and%

\begin{equation}
L_{3}(x,y,z):=\left(  1-\alpha\left\vert x\right\vert +y,\bar{\beta}%
z,1-\alpha\left\vert z\right\vert +\beta x\right)  ,\label{eq11}%
\end{equation}
where typically we have to choose $0<b/a,\bar{b}/a$ and $0<\beta/\alpha
,\bar{\beta}/\alpha$ considerably smaller than the corresponding ratio in
Theorem \ref{thm3} and Theorem \ref{thm4}, respectively, to compensate for the
lack of symmetry in the trapping regions, which would be unnecessary for the
GAH paradigm shown in Fig. 4. We note that it is interesting to compare
(\ref{eq11}) with the results in \cite{GMO}. Examples of attractors for
three-dimensional H\'{e}non and Lozi maps are shown in Fig.6 and
Fig.7, respectively. These attractors in $\mathbb{R}^{3}$ are
automatically CSAs, and one can infer from the figures that each has a
Hausdorff dimension somewhat greater than two.

\begin{figure}
\centering
\includegraphics[width = 0.9\textwidth]{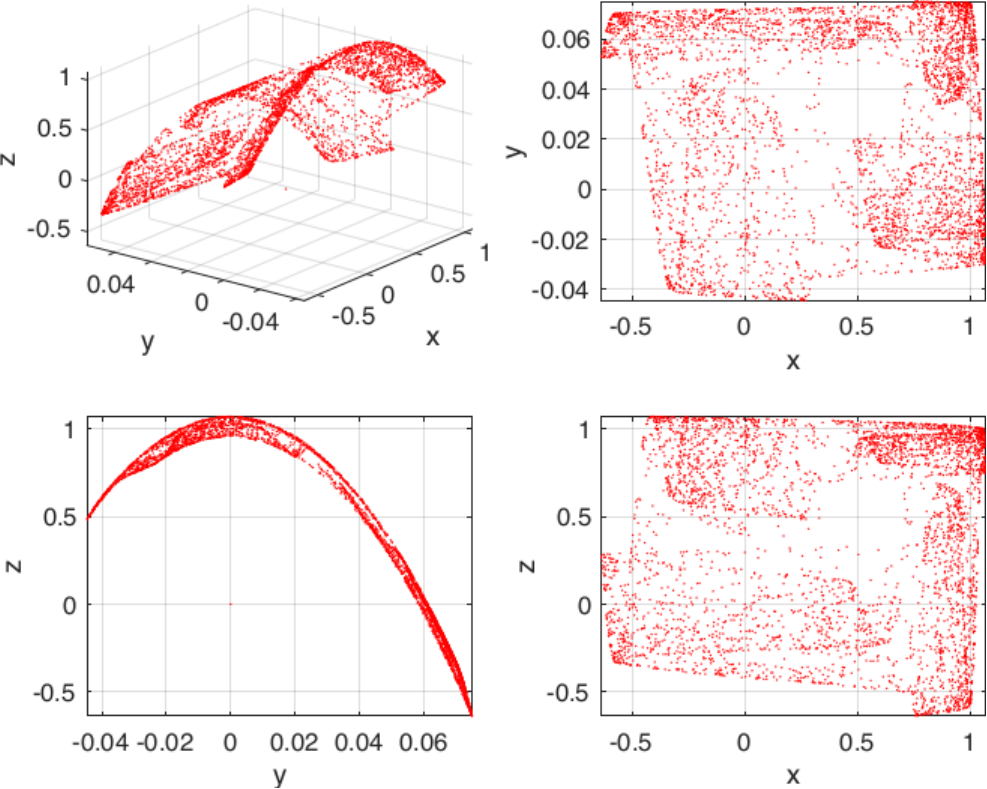}
\caption{Local Attractor for the map (\ref{eq10}) for $a=1.4$, $b=\bar{b}=0.09$.
Clockwise: full 3D; projections in $x,y$; $x,z$ and $y,z$ planes.}
\end{figure}


\begin{figure}
\centering
\includegraphics[width = 0.9\textwidth]{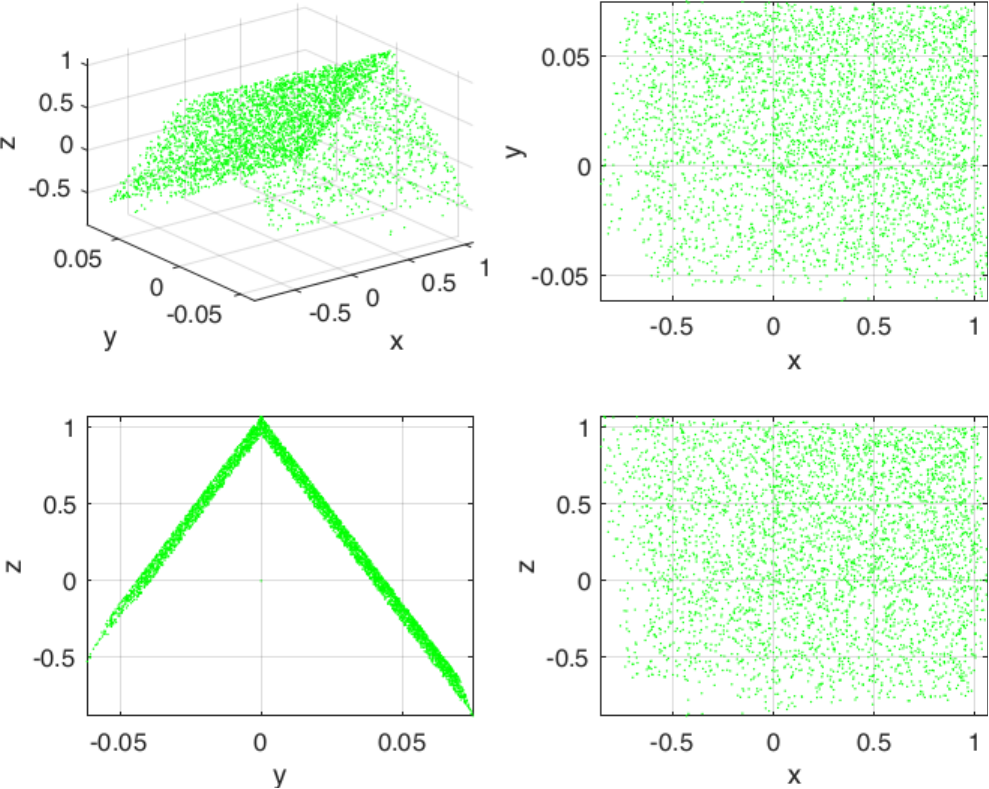}
\caption{Local Attractor for the map (\ref{eq11}) for $\alpha=1.7$, $\beta=\bar{\beta}=0.08$.
Clockwise: full 3D; projections in $x,y$; $x,z$ and $y,z$ planes.}
\end{figure}


\section{Concluding Remarks}

We have introduced GAH dynamical models proved that they have CSAs that are
closures of unstable manifolds of any finite dimension. Moreover, by showing
that they include the H\'{e}non and Lozi maps of the plane, we were able to
give comparatively simple proofs that they possess CSAs for certain parameter
values, and in the process resolved an apparently open question of rather long
standing about the existence of a CSA for the H\'{e}non map with parameters
$a=1.4$ and $b=0.3$. The models presented are those having 1-dimensional
stable manifolds for their key saddle points, so it is natural to also
consider models having higher dimensional stable manifolds for these points,
which we plan to do in the near future. In addition, along the lines in
\cite{BY,CL}, we intend to construct SRB measures for the GAH paradigms,
thereby enabling a deeper statistical study of their CSAs. Finally, we have
observed that our CSA constructions appear to have some important applications
in granular flow problems, and we intend to seek out additional areas of
science and engineering in which GAHs may be useful.

\section*{Acknowledgment}

\noindent Y. Joshi would like to thank a CUNY grant and his department for
support of his work on this paper, D. Blackmore is indebted to NSF Grant CMMI
1029809 for partial support of his initial efforts in this collaboration, and
A. Rahman appreciates the support of the DMS at NJIT for his participation.
Thanks are also due to Marian Gidea for insights derived from discussions
about chaotic strange attractors.

\qquad

\end{document}